\documentclass[a4paper,reqno,11pt,oneside]{amsart}
\usepackage{amsmath}
\usepackage[left=2.61cm,right=2.61cm,top=2.72cm,bottom=2.72cm]{geometry}
\usepackage[colorlinks=true,urlcolor=blue,
citecolor=red,linkcolor=blue,linktocpage,pdfpagelabels,
bookmarksnumbered,bookmarksopen]{hyperref}
\usepackage[english]{babel}
\usepackage{graphicx}
\usepackage[toc,page,header]{appendix}
\usepackage{mathrsfs}
\usepackage{amssymb,amsmath, amsfonts, amsthm}
\usepackage{latexsym}
\usepackage{enumitem}

\usepackage[latin1]{inputenc}

\allowdisplaybreaks[1]

\numberwithin{equation}{section}

\renewcommand{\(}{\left(}
\renewcommand{\)}{\right)}
\renewcommand{\[}{\left[}
\renewcommand{\]}{\right]}

\newtheorem{theorem}{Theorem}[section]
\newtheorem{proposition}[theorem]{Proposition}
\newtheorem{corollary}[theorem]{Corollary}
\newtheorem{lemma}[theorem]{Lemma}
\theoremstyle{definition}
\newtheorem{remark}[theorem]{Remark}
\theoremstyle{definition}

\theoremstyle{definition}

\renewcommand{\le}{\leqslant}
\renewcommand{\ge}{\geqslant}

\newcommand{\beq}{\begin{equation}}
\newcommand{\eeq}{\end{equation}}
\newcommand{\beqs}{\begin{equation*}}
\newcommand{\eeqs}{\end{equation*}}
\newcommand{\beqn}{\begin{eqnarray}}
\newcommand{\eeqn}{\end{eqnarray}}
\newcommand{\beqns}{\begin{eqnarray*}}
\newcommand{\eeqns}{\end{eqnarray*}}
\newcommand{\bdoc}{\begin{document}}
\newcommand{\edoc}{\end{document}}
\newcommand{\be}{\begin{enumerate}}
\newcommand{\ee}{\end{enumerate}}
\newcommand{\bdescr}{\begin{description}}
\newcommand{\edescr}{\end{description}}
\newcommand{\ba}{\begin{array}}
\newcommand{\ea}{\end{array}}
\newcommand{\intR}{\int_{\mathbb R^N}}
\newcommand{\R}{\mathbb R}
\newcommand{\RN}{\mathbb{R}^N}
\newcommand{\B}{\mathbb B}
\renewcommand{\H}{\mathcal H}
\renewcommand{\L}{\mathbb L}
\newcommand{\parallelsum}{\mathbin{\!/\mkern-5mu/\!}}
\newcommand{\e}{\varepsilon}
\newcommand{\SD}{\Sigma_D}
 \renewcommand{\(}{\left(}
\renewcommand{\)}{\right)}
\renewcommand{\[}{\left[}
\renewcommand{\]}{\right]}
\renewcommand{\appendixpagename}{\centering Appendix}

\newcommand{\todo}[1]{\text{\colorbox{yellow}{#1}}}
\newcommand{\add}[1]{\text{\color{red}{#1}}}

\begin{document}
\title[Some Remarks on Patterns for 
Semilinear Neumann Problems]{Some Remarks on Patterns for 
Semilinear Neumann Problems}

\author{Marta Calanchi}
\address{M. Calanchi. Dipartimento di Matematica ``Federigo Enriques'',
Universit\`a degli Studi di Milano, Via Cesare Saldini 50, 20133 Milano, Italy}
\email{marta.calanchi@unimi.it}

\author{Giulio Ciraolo}
\address{G. Ciraolo. Dipartimento di Matematica ``Federigo Enriques'',
Universit\`a degli Studi di Milano, Via Cesare Saldini 50, 20133 Milano, Italy}
\email{giulio.ciraolo@unimi.it}

\author{Francesca Messina}
\address{F. Messina. Dipartimento di Matematica ``Federigo Enriques'',
Universit\`a degli Studi di Milano, Via Cesare Saldini 50, 20133 Milano, Italy}
\email{francesca.messina@unimi.it}



\begin{abstract} We study semilinear elliptic equations
 \begin{equation*}
\begin{cases}
-\Delta u = f(u) & \text{in } \Omega, \\
\partial_\nu u = 0 & \text{on } \partial\Omega,
\end{cases}
\end{equation*}
 with homogeneous Neumann boundary conditions in bounded domains. A classical result by Casten-Holland and Matano shows that stable nonconstant solutions cannot exist in convex domains, although unstable spatial patterns may still occur.
 In this paper we investigate rigidity properties of classical solutions 
without imposing stability assumptions and aim to identify structural conditions on the nonlinearity ensuring that all solutions are constant.
We prove that every classical solution of the Neumann problem is constant provided the nonlinearity satisfies a suitable `monotonicity' condition, which includes the cases where the nonlinearity has a fixed sign or changes sign in a controlled way around one of its zeros. This yields a rigidity result depending solely on the structure of the nonlinearity and does not require convexity assumptions on the domain. 
We also discuss the sharpness of our assumptions by constructing examples of nonlinearities for which nonconstant solutions exist. In particular, inspired by the approach of Lin-Ni-Takagi, we consider  exponential-type nonlinearities in dimension $N=2$, and show that when a parameter crosses a critical threshold, the associated Neumann problem admits nontrivial and nonconstant solutions for sufficiently small diffusion.
\end{abstract}

\maketitle

\section{Introduction}

Semilinear elliptic equations with homogeneous Neumann boundary conditions arise naturally in several contexts, ranging from reaction--diffusion models to geometric problems. A fundamental question concerns the rigidity of solutions in convex/nonconvex domains, namely whether nontrivial solutions may exist under suitable structural assumptions on the nonlinearity.

\medskip
In this paper we are concerned with the following  Neumann problem.
Let $\Omega \subset \R^N$, with $N \geq 2$, be a bounded  domain with boundary of class $C^2$. We consider classical solutions to the semilinear problem

\begin{equation}\label{semilinear}
\left\{
\begin{array}{rll}
- \Delta u & = f(u) & \quad \text{ in } \Omega, \vspace{0.2cm} \\
{u}_{\nu} &= 0, 
& \quad \text{on } \partial\Omega,
\end{array}
\right.
\\
\end{equation}
where $\nu$ denotes the outward unit normal to $\partial\Omega$ and $f \in C^1(\R)$.
\\
It is clear that a zero of the nonlinearity $f$  (whenever it exists) is a constant solution of the problem.

%
\medskip

The question of whether nonconstant solutions of \eqref{semilinear}
can be excluded altogether has been widely investigated in the literature. 
In general, there is no simple necessary and sufficient condition depending 
only on the nonlinearity $f$ ensuring that all solutions are constant. 
Instead, available results provide various sufficient conditions involving 
the geometry of the domain, spectral properties of the Laplacian, and 
structural assumptions on $f$.

A fundamental result in this direction is due to Casten and Holland \cite{CH} 
and independently to Matano \cite{Ma}. They proved that if $\Omega$ is a 
bounded convex domain of $\mathbb{R}^n$, $n\ge2$, then every stable solution 
of \eqref{semilinear} must be constant. 

\smallskip
As customary, 
we say that a weak solution $u$ to \eqref{semilinear} is stable if 

\begin{equation}\label{stable}
\int_\Omega |\nabla \phi|^2 dx -\int_\Omega f'(u)\phi^2 dx\ge 0,\quad \forall \phi\in C^1(\Omega).
\end{equation}

\medskip

\begin{theorem}[Casten-Holland, Matano \cite{CH,Ma}]\label{CHM} Let $\Omega\subset\mathbb R^N$ be a convex domain. 
There exists no stable non-constant solution to \eqref{semilinear}.

\end{theorem}

In other words, convex domains 
do not support stable spatial patterns. However, this result does not exclude 
the existence of unstable nonconstant solutions. 

\medskip

If we call a \emph{pattern} any non-constant solution to \eqref{semilinear}, the  result mentioned above states that stable patterns cannot exist in convex domains. Besides its intrinsic mathematical interest, this theorem has significant implications for the classification of solutions and for the analysis of the asymptotic behavior of the associated evolution problems.

To rule out nonconstant solutions completely, further assumptions are typically 
required. Several works establish nonexistence results under conditions involving 
the structure of the nonlinearity, geometric properties of the domain, or 
smallness assumptions on the nonlinearity itself. In particular, criteria have 
been obtained showing that when the nonlinearity is sufficiently small or when 
certain geometric constraints are satisfied, the only solutions of the Neumann 
problem are constants.

For example, work by  Nordmann (\cite{No}) provides criteria for the nonexistence of nonconstant solutions involving
the geometry of the domain
and bounds on the size of the nonlinearity $f$. 
These results show that under suitable geometric conditions and smallness of the nonlinearity, no patterns exist at all.

%
%

\medskip
Another common approach relies on the spectral properties of the Neumann Laplacian. 
Let $\lambda_2(\Omega)$ denote the first nonzero Neumann eigenvalue of $-\Delta$ in 
$\Omega$. Roughly speaking, if the derivative $f'(u)$ remains below this spectral 
threshold, bifurcation from constant solutions cannot occur and one expects that 
all solutions are constant. Conversely, when the derivative of the nonlinearity 
crosses a Neumann eigenvalue, branches of nonconstant solutions may appear through 
bifurcation phenomena (see e.g \cite{CR1}).

\medskip

The appearance of patterns  has also been extensively studied. 
A classical example is provided by the work of Lin, Ni and Takagi (\cite{LNT}) on semilinear 
Neumann problems with nonlinearities of logistic or power type, where spike-layer 
solutions can arise under suitable parameter regimes (see also \cite{GPW} and the references therein). These results illustrate 
how the interplay between the structure of $f$, the parameters of the problem, 
and the geometry of the domain may lead to the formation of spatially nonhomogeneous 
solutions.

Taken together, these results show that the presence or absence of nonconstant 
solutions for \eqref{semilinear} seems to be  governed by a delicate interaction 
between the geometry of the domain, the spectral properties of the Neumann 
Laplacian, and the qualitative structure of the nonlinearity.

\medskip
%
All these results have been extended in various directions. In particular, nonlinear elliptic operators, different boundary conditions, manifolds, unbounded or more general domains, as well as certain classes of systems have been considered. We refer to \cite{BMMP, BPT, DPV, DPV2} and the references therein for further developments.
%


\smallskip
The assumptions of Theorem \ref{CHM} appear to be essentially optimal in general. On the one hand, the conclusion of the theorem does not hold for unstable patterns. 
 A first simple example is provided by Neumann eigenfunctions of the Laplacian in a disk, corresponding to the linear choice $f(u)=\mu u$ with $\mu$ a positive Neumann eigenvalue. 
This phenomenon is also closely related to Schiffer's conjecture, which states that disks are the only bounded domains whose Neumann eigenfunctions are constant along the boundary.

 
  On the other hand, the assumption that the domain is convex cannot be removed. Indeed, a result by Kohn and Sternberg~\cite{KS} shows that stable patterns can be constructed in certain star-shaped domains obtained as small perturbations of a convex domain. A similar counterexample can also be produced in dimension $N \ge 3$ for domains with positive mean curvature.

Nevertheless, when the domain is close to being convex, the construction of a stable pattern requires the nonlinearity to have a large amplitude. This observation suggests that, in non--convex domains, the nonexistence of patterns depends both on the geometry of the domain and on the strength of the nonlinearity.

\medskip

In \cite{CCR1}, Ciraolo-Corso-Roncoroni obtained a result in the spirit of the classical works of Casten--Holland and H. Matano by considering general weak solutions to \eqref{semilinear} in dimension $N\ge3$. They removed the stability assumption and proved that constant functions are the only weak solutions to \eqref{semilinear}, provided that the nonlinearity satisfies the following condition:
\begin{equation*} \label{f_condition}
t\mapsto \frac{f(t)}{t^{\frac{N+2}{N-2}}}
\quad \text{is non-increasing}.
\end{equation*}
This assumption seems, in a certain sense, optimal. Indeed, counterexamples can be constructed by adding a small linear perturbation to $f$ (see e.g.  \cite{CCR1}).

\medskip

The aim of this paper is to identify conditions depending solely on the nonlinearity that guarantee either the existence or the nonexistence of patterns. To the best of our knowledge, the existing literature does not provide results in this direction. In particular, no sufficient condition ensuring the non-existence of patterns in non--convex domains is currently known.

\medskip

%
 
 In this paper, we tackle this issue and establish a rigidity result showing that every classical solution must be constant, provided that the nonlinearity satisfies an appropriate monotonicity assumption. More precisely, our aim is to mimic the aforementioned comparison with the critical growth condition \eqref{f_condition}, adapting it to the present framework in order to derive the desired constancy result. As it turns out, however, this comparison with the critical growth condition is not essential.
%
%
%
 \smallskip
 
%
%
 
 
 \begin{theorem} \label{thm_main1}
Let $f \in C^1(\R)$ be a function that changes sign.
%
Let $\xi$ be a root of $f$ such that  
\begin{equation} \label{f_cond}
f(t)(t-\xi)\le 0, \quad \forall t\in\mathbb R.
\end{equation}
Then any classical solution  of 
 \eqref{semilinear} must be  constant in $\overline \Omega$.
\end{theorem}
This result has an immediate consequence.

\begin{corollary}\label{corollary} Let $f \in C^1(\R)$ and suppose $f(t)=p(t) \varphi(t)$ where 
 $p(t)> 0$  and 
$
 \varphi(t)$  is nonincreasing.  Then, there exists no  non-constant solution to \eqref{semilinear}.

\end{corollary}

We would like to emphasize that the result does not require any growth assumptions on the nonlinearity, nor does it require any bound on its derivative. We also mention that if $f$ has constant sign, then the result immediately follows by integrating \eqref{semilinear} in $\Omega$, applying the divergence theorem and the using the boundary condition $u_\nu =0$.

\smallskip


Our result completes previous works (\cite{CCR1, CCR2}) in  dimension $N\ge3$, where analogous rigidity statements were obtained under stronger hypotheses, such as the positivity of the solution and different structural conditions on the nonlinearity. In contrast, we are able to remove the positivity assumption and treat the full range of classical solutions.

As already discussed, in the absence of suitable assumptions on 
$f$ nonconstant solutions may arise.
 The case of eigenfunctions is not the only counterexample that can be presented. Even in the classical problem with polynomial-growth nonlinearity (e.g. the Emden-Fowler equation), there are precise results concerning the existence or non-existence of patterns. We refer, in particular, to the works of Lin, Ni, and Takagi (\cite{Lin-Ni,LNT}). In \cite{LNT}, the authors establish the following result:
\begin{theorem}[Lin-Ni-Takagi, \cite{LNT}]
Let $\Omega \subset \R^N$ be a bounded domain with boundary of class $C^2$. We consider classical solutions to the semilinear problem

\begin{equation}\label{polynomial}
\left\{
\begin{array}{rll}
- \epsilon\ \Delta u+u & = u^p & \quad \text{ in } \Omega, \vspace{0.2cm} \\
u&> 0, 
& \quad \text{ in } \Omega, \vspace{0.2cm}
\\
{u}_{ \nu} &= 0, 
& \quad \text{ on } \partial\Omega,
\end{array}
\right.
\\
\end{equation}
where $1<p<\frac{N-2}{N+2}$ for $N\ge 3$ and any $p>1$ for $N=1,2$. Then there exists $\epsilon_0>0$ such that for every $0<\epsilon<\epsilon_0$ there exists a non constant classical solution to \eqref{polynomial}, while for  large values od $\epsilon$, there is only the constant solution $u\equiv 1$.
\end{theorem}
 For simplicity, we stated it here in the case of the polynomial nonlinearities $u\mapsto u^p$, although it extends to more general classes of nonlinear terms with polynomial growth. 

\medskip

\medskip 
In dimension $N=2$ every polynomial growth is allowed. Therefore, it seems that the ``critical'' growth is given by the exponential nonlinearity (which is also connected to the Liouville equation).  Drawing inspiration from the works of Lin, Ni, and Takagi, and following their approach, we  provide a family of nonlinearities for which condition \eqref{f_cond} in Theorem \ref{thm_main1} does not hold, and we prove the existence of nonconstant solutions to \eqref{semilinear}.

 Consider the  problem
\begin{equation}\label{semilinearexp}
\left\{
\begin{array}{rll}
-\epsilon \Delta u &= f_a(u) & \textmd{ in } \Omega \\
 u &> 0 & \textmd{ in } \Omega \\
u_\nu &= 0 & \textmd{ on } \partial \Omega \,,
\end{array}
\right.
\\
\end{equation}
where 
$
f_a(t):= e^t-1-a t 
$
 and  $a \in \R$ is a parameter.
We have the following result

\begin{theorem}\label{non_constant} For every $a\le 1$ the problem admits only the trivial solution, while for every $a>1$ there exists $\epsilon_0>0$ such that 
 for every $\epsilon\in(0,\epsilon_0)$ problem \eqref{semilinearexp} admits at least a non constant solution.

\end{theorem}


\medskip
%
The proof of the existence of non-constant patterns relies on the search for critical points of the functional associated with the problem, which is well defined on $H^1(\Omega)$. For $a>1$, and for every $\epsilon>0$, the problem posses the trivial solution and a positive constant solution $u\equiv\xi$. Following the ideas in \cite{LNT},  we establish the existence of a solution whose energy level is positive and strictly  lower than that of the constant solution $u\equiv\xi$.

\bigskip

\section{Proof of Theorem \ref{thm_main1}}\label{section2}
As already mentioned in the introduction, we first observe that a condition on the nonlinearity 
 that trivially guarantees the nonexistence of patterns is when the function $f$
does not change sign, i.e. one of the following condition is satisfied

 \begin{equation}\label{f_sign}f(t) \geq 0\quad {\rm  or}\quad  f(t) \leq 0, \ \forall t \in \R.\end{equation}

\smallskip
In this case, it is sufficient to integrate both sides of the equation over 
$\Omega$
 and use the boundary conditions. Indeed, let $u $ be  a solution of \eqref{semilinear},  with\  for instance $ f(u)\ge 0$. Then,
\medskip\noindent
$$
0=\int_\Omega -\Delta u  = \int_\Omega f(u)\ dx.$$ 
Since $f(u)\ge 0$ the identity implies that $ f(u)=0\ {\rm a.e.}$ which implies $ u={\rm constant}.$
\\
In particular, if the function $f$ has no zeros, then there are no solutions at all. If $f $ has zeros, then those are the only solutions to the problem.

\medskip
\begin{proof}[{Proof of Theorem \ref{thm_main1}}]


\medskip

\noindent
 Let $u$ be a solution of \eqref{semilinear}, and let $\xi\in\mathbb R$ be such that $f(\xi)=0$ and $f(t)(t-\xi) \le 0$. If we multiply by $u-\xi$ and integrate the equation, using the Neumann boundary conditions, we obtain, since 
$f(t)(t-\xi) \le 0$ for all $t\in\mathbb R$,

\begin{equation}
0 \le \int_\Omega|\nabla u|^2 dx=\int_\Omega f(u(x))(u(x)-\xi) dx\le 0.
\end{equation}
which implies that $u$ is constant and $f(u)(u-\xi)=0$ a.e., that is $u$  must be a root of $f$.
\end{proof}

\begin{remark}
In this case, any constant solutions are all stable, since if $\eta\in\mathbb R$ is a solution, it is a zero of $f$ with $f'(\eta)\le 0$.
\end{remark}

\begin{proof}[{Proof of Corollary \ref{corollary}}]
In the particular case $f(t)=\varphi(t)p(t)$, where $p>0$ and $\varphi$ non-increasing, the condition \begin{equation}
f(t)(t-\xi)=p(t)\varphi(t)(t-\xi)\le 0, \quad \forall t\in\mathbb R.
\end{equation}
is automatically satisfied, and Theorem \ref{thm_main1} applies.

\smallskip
\end{proof}

\section{Existence of positive patterns in generic domains in $\mathbb R^2$}
As pointed out in the introduction, in the absence of appropriate assumptions on $f$,  patterns solutions may appear. Eigenfunctions are not the only source of counterexamples. Indeed, even for the classical problem involving nonlinearities with polynomial growth -- such as the 
Emden-Fowler equation -- there exist sharp results addressing both the existence and the nonexistence of nontrivial patterns. 
 Motivated by the ideas developed by Lin, Ni, and Takagi (\cite{Lin-Ni,LNT}), and adopting a similar strategy,  we consider in dimension $n=2$, exponential-type nonlinearities 
 for which conditions \eqref{f_sign} or \eqref{f_cond}
 of Theorem \ref{thm_main1} fail, and we establish the existence of nonconstant solutions to \eqref{semilinear}.
%
%

\medskip
Let 
\begin{equation} \label{f_a}
f_a(t):= e^t-1-a t 
\end{equation}
for $t \in \R$ and where $a \in \R$ is a parameter.

 It is readily seen that for $a \leq 1$ we have $f_a \geq 0$, while for $a > 1$ the function $f_a$ does not have a sign. Hence, $f_a$ does not satisfy condition \eqref{f_sign}  whenever $a>1$. 
 
 Thus, only the constant solution $u=0$ exists when $a=1$. By letting $a=1+\delta$, with $\delta>0$ small, $f_a$ can be seen as a perturbation of the case $a=1$. The problem
\begin{equation*}
\begin{cases}
-\Delta u = e^u - 1 - (1+\delta) u & \textmd{ in } \Omega \\
u  > 0 & \textmd{ in } \Omega \\
u_\nu  = 0 & \textmd{ on } \partial \Omega \,.
\end{cases}
\end{equation*}
 admits  two constant solutions: $u\equiv 0$ and $u\equiv\xi$, for some $\xi=\xi(\delta) >0$. It is convenient to introduce a further parameter $\epsilon>0$ and consider the following problem
\begin{equation}\label{pbeps}
\left\{
\begin{array}{rll}
-\epsilon \Delta u  &= e^u - 1 - (1+\delta) u & \textmd{ in } \Omega \\
u  &>0 & \textmd{ in } \Omega \\
u_\nu  &= 0 & \textmd{ on } \partial \Omega \,.
\end{array}\\
\right.
\end{equation}

We adopt  a variational approach. The equation above is
  the Euler-Lagrange equation associated to the functional
\begin{equation} \label{Phi_eps}
\Phi_\epsilon (u) = \frac{\epsilon}{2} \int_\Omega |\nabla u|^2 dx  - \int_\Omega F_\delta(u) dx \,,
\end{equation}
where 
\begin{equation} \label{F_delta}
F_\delta(u) = e^u - 1- u-(1+\delta) \frac{u^2}{2} \,.
\end{equation}
Notice that the functional $F_\delta$ is well defined on the Sobolev space $H^1(\Omega)$, endowed with the norm
\begin{equation*}
\|u\|_*^2:= \epsilon \int_\Omega |\nabla u|^2 + \delta \int_\Omega u^2
\end{equation*}
which is a norm equivalent to the standard $H^1$ norm $\|\cdot\|_{H^1}$.

We prove the following

\begin{proposition}\label{estimate}
For every $\delta>0, \ \epsilon>0$ there exists a  solution $u_\epsilon$ to \eqref{pbeps}
that 
 satisfies 
\begin{equation}
0<\Phi_\epsilon (u_\epsilon) < C\epsilon.
\end{equation}
for some constant $C=C(\\lOmega,\delta). $ 
\end{proposition}

Theorem \ref{non_constant} follows easily from the following elementary lemma.

\begin{lemma}

Again, for any $\epsilon>0$, we have that $u=0$ and $u=\xi$ are solutions to \eqref{pbeps} and that $\Phi_\epsilon(0)=0$ and $\Phi_\epsilon(\xi) = K_\xi|\Omega |$, where $K_\xi$ is a {\it positive}  constant independent of $\epsilon$. 
\end{lemma}
\begin{proof}Let  $\xi >0 $, such that  $e^\xi-1-(1+\delta)\xi=0$. We have 
$$
K_\xi=\frac{1+\delta}{2}\, \xi^2-(e^\xi-1-\xi),
\quad {\rm and}\quad
(1+\delta)=\frac{e^\xi-1}{\xi}.
$$
Therefore,
$$
K_\xi=\frac{e^\xi}{2}\, \xi-e^\xi+1+\frac{\xi}{2}>0.
$$
Indeed, let $h(x)=\frac{e^x}{2}\, x-e^x+1+\frac{x}{2}$; then $h(0)=h'(0)=h''(0)=0$ and $h''(x)>0$ for all $x>0$.

\end{proof}

\begin{proof}[Proof of Proposition \ref{estimate}]
The existence of  $u_\epsilon$  will be proved by using the Mountain Pass Lemma by Ambrosetti and Rabinowitz \cite{AR}.

\medskip
\emph{Step 1: (Local minimum in $u=0$)}.  We first define the norm 
\begin{equation*}
\|u\|_*^2:= \epsilon \int_\Omega |\nabla u|^2 + \delta \int_\Omega u^2.
\end{equation*}
 Now, we have that, as $u\to0$ in $H^1$,
\begin{equation} \label{zero_min}
\Phi_\epsilon(u)=\frac12 \|u\|_*^2 + o(\|u\|^2_{H^1}),
\end{equation}
 which implies that $u=0$ is a strict local minimum. In particular, there exist  $\rho>0$ and $\alpha>0$ such that 
 \begin{equation}
 \Phi_\epsilon(u)\ge \alpha\epsilon,
\end{equation}
for any $u$ such that $\|u\|_*= \rho$.

 \medskip\noindent
 The proof of \eqref{zero_min} is fairly standard and relies on the fact that the nonlinearity $$g(t)=e^t-1-t-\frac{t^2}{2}=o(t^2)\quad
{\rm as}\quad  t \to 0.$$
For completeness we give here a sketch of the proof. In fact, we show that there exists a constant $C=C(\Omega)>0$ s.t.
\begin{equation}\label{zero_min2}
\Big{|} \int_\Omega \left(e^u-1-u-\frac{u^2}{2} \right)\Big{|} \leq C \|u\|_{H^1}^3 \,,\quad {\rm as}\ \  u\to 0.
\end{equation}

By the Taylor series expansion of the exponential, we have
$$
e^u-1-u-\frac{u^2}{2} = \sum_{k=3}^{+\infty} \frac{u^k}{k!}
$$
and thus it will be enough to prove that 
\begin{equation}\label{zero_min3}
\sum_{k=3}^{+\infty} \int_\Omega\frac{|u|^k}{k!} \leq  C \|u\|_{H^1}^3 , \ u\to 0\ \ {\rm in} \ H^1,
\end{equation}
 which will be proved as a consequence of the following estimate:
\begin{equation} \label{fede}
\|u\|_k \leq C\left(\Gamma\left( \frac{k}{2} + 1\right)\right)^{\frac{1}{k}} \|u\|_{H_1} \,,
\end{equation}
for some $C$ independent of $k$. We notice that, by symmetrization, it is enough to prove \eqref{fede} for radial functions defined in a ball $B_R$ centered at the origin. Hence, by assuming $u(x) = \tilde u (r)$ for $r=|x|$, we have that 
$$
v(r) = \tilde u(r) - \tilde u (R)
$$
is a radial function in $H_0^1(B_R)$ and 
\begin{equation} \label{nablavu}
\int_{B_R} |\nabla v|^2 = \int_{B_R} |\nabla \tilde u|^2 \,.
\end{equation}
Since $\tilde u \in H^1$ is radial,  we have (see \cite{Ruf})
\begin{equation} \label{ruf}
|\tilde u(r)| \leq \frac{1}{r\sqrt{\pi}} \|\tilde u\|_{L^2}
\end{equation}
for any $r>0$. For $v \in H_0^1(B_R)$, it yields (see e.g \cite{CT}) 
\begin{equation} \label{calaterra}
\|v\|_k \leq  C\left(\Gamma\left( \frac{k}{2} + 1\right)\right)^{\frac{1}{k}} \|\nabla v\|_{2},
\end{equation}
where the constant $C=C(\Omega)$ depends only on $\Omega$, and $\Gamma$ is the {\it Gamma-function}.

\smallskip
Since $\tilde u(r) = v(r) + \tilde u(R)$, from Minkowski inequality and using \eqref{ruf} and \eqref{calaterra} we obtain
\begin{equation*}
\begin{split}
\|\tilde u\|_k & \leq \|v\|_k + \left( \int_{B_R} |\tilde u(R)|^k \right)^{\frac{1}{k}}  \\
& \leq C\left(\Gamma\left( \frac{k}{2} + 1\right)\right)^{\frac{1}{k}} \|\nabla v\|_{2} + \frac{1}{R\sqrt{\pi}} |B_R|^{\frac{1}{k}} \|\tilde u\|_{L^2} 
\end{split}
\end{equation*}
and, from 
$$
\lim_{k\to +\infty} \left(\Gamma\left( \frac{k}{2} + 1\right)\right)^{\frac{1}{k}} = + \infty
$$
and \eqref{nablavu} we obtain 
\begin{equation} \label{ukbound}
\|\tilde u \|_k \leq C \left(\Gamma\left( \frac{k}{2} + 1\right)\right)^{\frac{1}{k}}  \|\tilde u\|_{H^1(B_R)} \,.
\end{equation}
By symmetrization, \eqref{ukbound} holds also for $u\in H^1$ (not necessarily radial). Hence
$$
\sum_{k=3}^{\infty} \frac{\|u\|_k^k}{k!} \leq \sum_{k=3}^{\infty} \frac{C^k}{k!} \Gamma\left( \frac{k}{2} + 1\right) \|u\|_{H^1}^k \leq \bar C  \sum_{k=3}^{\infty} \|u\|_{H^1}^k \,,
$$
and then
$$
\sum_{k=3}^{\infty} \frac{\|u\|_k^k}{k!} \leq C \|u\|_{H^1}^3 
$$
for $\|u\|_{H^1}^3\leq 1/2$, which is \eqref{zero_min3}. This implies \eqref{zero_min2} which proves \eqref{zero_min}. Thus we have 

\begin{equation}\label{Phiupiccolo}
 \frac12 \min(\epsilon, \delta) \|u\|_{H^1}^2 + o(\|u\|_{H^1}^2) \leq \Phi_\epsilon(u) \leq \frac12 \max(\epsilon^{-1}, \delta^{-1}) \|u\|_{H^1}^2 + o(\|u\|_{H^1}^2)
\end{equation} 
for $\|u\|_{H^1}^2 \to 0$. The first inequality in \eqref{Phiupiccolo} yields that there exist $\rho>0$ and $\alpha>0$ such that

\begin{equation}\label{alpha}
\Phi_\epsilon(u) \geq \alpha \epsilon>0 \quad \textmd{ for }  \ \|u\|_{H^1} = \rho \,.
\end{equation} 

\bigskip
\emph{Step 2: there exists $w$ with $\|w\| \geq \rho$ such that $\Phi_\epsilon(w) \leq 0$}. This step is easily achieved by evaluating the functional on constant functions. Indeed, let $K \in \R$, we have
$$
\Phi_\epsilon(K) = \frac{\delta}{2}\int_\Omega K^2 - \int_\Omega \left(e^K-1-K-\frac{K^2}{2}\right) \to -\infty \quad \textmd{ as  } K \to +\infty \,.
$$
Now we choose  $w=\bar K$ such that $\Phi_\epsilon(\bar K) <0$ and $\| \bar K\|\ge \rho$.

From \emph{Step 1} and \emph{Step 2} we have that the energy functional $\Phi_\epsilon$ satisfies the geometric hypotheses of the Mountain Pass Lemma.

\bigskip

\emph{Step 3:  (The functional satisfies the Palais-Smale condition).} 
This follows directly from the subcritical nature of the nonlinearity 
 $g(t)=e^t-1-t-\frac{t^2}{2}$  in two dimensions.

\medskip
Therefore, all the hypotheses  of the Mountain Pass Lemma are satisfied. Let 

$$\Sigma :=\{ \gamma\in C([0,1], H^1):\ \gamma (0)=0, \gamma (1)=\bar u,\, \Phi_\epsilon(\bar u)<0, \ ||\bar u||_{H^1}>\rho \}\neq \emptyset\quad{\rm from} \ Step\ {\it 2}.$$
Then
\begin{equation} c_\epsilon := \inf_{\gamma\in \Sigma}\, \sup_{u\in \gamma([0,1])} \Phi_\epsilon (u),
\end{equation} 
 is a critical level for $\Phi_\epsilon$.
\\

Therefore, there exists  $u_\epsilon $ a positive critical point  (called  ground-state solution)  corresponding to the critical value\ $c_\epsilon$, i.e. $\Phi_\epsilon(u_\epsilon)=c_\epsilon$.

\medskip

\emph{Step 4: ($\Phi_\epsilon(u_\epsilon)=c_\epsilon\asymp \epsilon$ as $\epsilon\to 0^+$): } we need to prove that there exist two constants $d_1$ and $d_2$ such that 
$$
d_1\epsilon\ \le\  c_\epsilon\le\  d_2\epsilon.
$$
The estimate on the left is a direct consequence of \eqref{alpha}, with $d_1=\alpha$. 

\smallskip\noindent Indeed,  $\sup_{u\in \gamma([0,1])} \Phi_\epsilon (u)\ge \alpha \epsilon$, since every path $\gamma$ must intersect the ball $||u||=\rho$.

\medskip
For the estimate on the right we follows the argument in \cite{LNT}.

\smallskip\noindent
Without loss of generality, suppose now that $0\in \Omega$ and consider the following family of functions
$$
\omega_\epsilon (x):=
\left\{
\begin{array}{lcl}
	\displaystyle \epsilon^{-1}\Big( 1-\epsilon^{-\frac{1}{2}}|x|\Big),\,\, |x|\leq \sqrt\epsilon ,\\
	& & \\
	\displaystyle 0, \,\,\,\,\,\,\,\,\,\,\,\,\,\,\,\,\,\,\,\,\,\,\,\,\,\,\,\,\,\,\,\,\,\,\,\,\,\,\,\,\,\,\,|x|> \sqrt\epsilon.& \;\;
\end{array}\right.
$$
Then, for all $k\in \mathbb{N}$,
$$
\int_\Omega |w_\epsilon |^k\, dx=c_k\epsilon ^{1-k}
$$
where 
$$
c_k=2\pi \int_0^1\, (1-\rho)^k\rho\, d\rho =2\pi B(2,k+1)=\frac{2\pi}{(k+2)(k+1)}.
$$
and $B(\cdot,\cdot) $ is the $\beta eta$-function. 

\medskip
We have the following identities, proved in \cite{LNT},
$$
\int_\Omega |w_\epsilon |^k\, dx=\frac{2\pi}{(k+2)(k+1)}\epsilon ^{1-k},
\\
\int_\Omega |\nabla w_\epsilon |^k\, dx=\pi \epsilon ^{-2}.
$$
We evaluate the functional $\Phi_\epsilon$ on the half-line $x=tw_\epsilon$, for all $t>0$:
$$
g_\epsilon(t):=\Phi_\epsilon (tw_\epsilon)=\frac{\pi}{2\epsilon}\bigg( 1+\frac{\delta}{6}\bigg)t^2- 
2\pi\epsilon \sum_{k=3}^{+\infty}\, \frac{1}{(k+2)!}\bigg(\frac{t}{\epsilon}\bigg)^k.
$$

To conclude it is sufficient to prove that

\medskip
\begin{enumerate}

\item[$(g_1)$]\  $g_\epsilon(t)$ has a minimum in $t=0$;

\medskip

\item[$(g_2)$]\  $g_\epsilon(t)\to-\infty$, as $t\to+\infty$;

\medskip

\item[$(g_3)$] \ $g_\epsilon$ has a maximum and $\displaystyle \max_{t\ge 0}g_\epsilon(t)\le C_2\epsilon.$
\end{enumerate}

\medskip
Note that - from $(g_1)$ and $(g_2)$ - the function $t\mapsto t w_\epsilon$ - up to re-parametrization - is an ammissible path belonging to the class $\Sigma$ defined above.

\medskip
%
We have that  $
g_\epsilon(t)=\eta \big( {t}/{\epsilon}\big),$ where
 $$\eta(x) :=\frac{\pi}{2}\Big(1+\frac{\delta}{6}\Big)\epsilon x^2- \frac{2\pi \epsilon}{x^2}\Big(e^x-1-\frac{x^2}{2}-\frac{x^3}{3!}-\frac{x^4}{4!}\Big), \, x\in [0,+\infty).$$ 
\\

It follows easily that $$\eta(x) =\frac{\pi}{2}\Big(1+\frac{\delta}{6}\Big)\epsilon x^2+o(x^2), \ \  {\rm as}\ \ x\to 0^+, \ {\rm and}\quad 
\eta(x)\to -\infty, \quad {\rm as} \ \ x\to +\infty,
$$
\\
so that $$g_\epsilon(t)=\frac{\pi}{2\epsilon}\Big( 1+\frac{\delta}{6}\Big) t^2+o(t^2) \quad  {\rm as}\ \  t\to 0^+, \ \ {\rm and}\quad 
 g(t)\to -\infty,  \quad {\rm as} \ \ t\to +\infty.$$ Therefore, $t=0$ is a local minimum for $g_\epsilon$ and $g_\epsilon$ has an absolute maximum point in $(0,+\infty )$.  
\\
Now we are left to prove that  there exists $C_2>0$ such that $\max_{t\in \mathbb{R}^+}\, g_\epsilon(t)\leq C_2\epsilon$.
\\
 First we observe that (here  $A=\frac{\pi}{2}\Big(1+\frac{\delta}{6}\Big)$ and  $B=2\pi$)
 $$
\eta (x)\leq \epsilon(A x^2-B \frac{x^3}{5!})=\epsilon h(x),
$$
\\
since
$$
e^x-\sum_{j=0}^{4}\, \frac{x^j}{j!}\geq \frac{x^{5}}{5!},\quad  \forall x\in [0,+\infty).
$$
\\
Therefore, $$\max_{x\in (0,+\infty)}\, \eta(x)\leq \epsilon\max_{x\in (0,+\infty)}\, h(x)=C_2\epsilon$$
This is sufficient to conclude the proof since the maximum of $h$ does not depend on $\epsilon$
$$
\max_{t\in (0,+\infty)}\, g_\epsilon(t)=\max_{x\in (0,+\infty)}\, \eta(x)\le C_2\epsilon.
$$
Anyway, we can give an estimate of the constant $C_2$. Indeed, 
it is easy to prove that
$$
C_2=\max_{x\in (0,+\infty)}\, h(x)=A \bigg(\frac{2A5!}{3B}\bigg)^2-\frac{B}{5!}\bigg(\frac{2A5!}{3B}\bigg)^3 =
\frac{6400 A^3}{3B^2}.
$$

\end{proof}


\subsection{Some final remarks}
With a slight modification, the result can be extended to more general nonlinearities behaving like $f_a$. 

\medskip

Another  interesting question is whether, for large values of the diffusion parameter $\epsilon$, only constant solutions can exist, as happens in the polynomial problem  (see e.g. \cite{NT}). The argument relies on a fundamental a priori estimate for the solutions. In the exponential case, however, such an a priori estimate does not generally hold. Despite this, we believe that for large values of 
$\epsilon$ only constant solutions actually persist. We therefore formulate the following conjecture.

\medskip
{\bf Conjecture.} {\it
There exists $\bar \epsilon$ such that for every $\epsilon>\bar\epsilon$  the Problem \ref{pbeps} admits only the constant solutions.}

\section*{Acknowledgments} \noindent {\small The first two authors are members of the Gruppo Nazionale per l'Analisi Matematica, la Probabilit\`a e le loro Applicazioni (GNAMPA) of the Istituto Nazionale di Alta Matematica (INdAM, Italy) and have been partially supported by the INdAM - GNAMPA Project 2026, CUP \#E53C25002010001\#.}

\end{document}